\documentclass{article}
\usepackage[a4paper,body={16cm,26cm}]{geometry}

\usepackage{authblk}
\usepackage{graphicx}
\usepackage{subcaption}
\usepackage[pdf]{graphviz}
\usepackage{amsmath}
\usepackage{amssymb}
\usepackage{stmaryrd}
\usepackage{listings}
\usepackage{multirow}
\usepackage{xcolor}
\usepackage{comment}
\usepackage{url}
\usepackage{amsthm}
\newtheorem{proposition}{Proposition}
\newtheorem{lemma}{Lemma}

\usepackage{tikz}
\usetikzlibrary{calc}

\newcommand{\bzero}{\mathbf{0}}
\newcommand{\mat}[2]{\mathcal{M}_{#1,#2}(\R)}
\newcommand{\iset}{\llbracket 1, N \rrbracket}
\newcommand{\iiniset}{\forall i \in \iset}

\newcommand{\Sset}{\mathcal{S}}
\newcommand{\SprojX}{\Pi_{\Xset}(\Sset)}

\newcommand{\prob}{\mathbb{P}}
\newcommand{\ba}{\mathbf{a}}
\newcommand{\bB}{\mathbf{B}}
\newcommand{\bC}{\mathbf{C}}
\newcommand{\balpha}{\boldsymbol{\alpha}}
\newcommand{\bbeta}{\boldsymbol{\beta}}
\newcommand{\bx}{\mathbf{x}}
\newcommand{\by}{\mathbf{y}}
\newcommand{\bz}{\mathbf{z}}
\newcommand{\R}{\mathbb{R}}
\newcommand{\sgn}{\textrm{sgn}}
\newcommand{\defspace}[1][]{\R^{d\ifx&#1&\else_{#1}\fi}}
\newcommand{\Xset}{\mathcal{X}}
\newcommand{\Yset}{\mathcal{Y}}

\newcommand{\XYset}{\Xset \times \Yset}

\def\projalpha{-0.6}
\def\projbeta{-0.4}
\newcommand{\pt}[4]{%
  \pgfmathsetmacro{\px}{#2 + \projalpha*(#4)}%
  \pgfmathsetmacro{\py}{#3 + \projbeta*(#4)}%
  \coordinate (#1) at (\px,\py);
}

\begin{document}

\title{A methodology for creating multidisciplinary design optimization benchmark problems from optimization ones}

\author[1]{Matthias De Lozzo}
\author[2]{Olivier Roustant}
\author[ ]{Amine Aziz-Alaoui}
\affil[1]{Institut de Recherche Technologique Saint Exupéry, France}
\affil[2]{UMR CNRS 5219, Institut de Mathématiques de Toulouse, INSA, Université de Toulouse, France}

\maketitle

\begin{abstract}

Benchmark problems with known solutions play a central role in the assessment of optimization algorithms. While mono-disciplinary optimization benefits from a rich collection of such problems, multidisciplinary design optimization (MDO) lacks equivalent resources: existing MDO benchmarks are scarce, rarely scalable, and their solutions are generally not known theoretically. In this paper, we propose a systematic methodology to transform any mono-disciplinary optimization problem with a known solution into a family of parametric MDO problems sharing that same solution. The construction relies on two key ingredients: a set of coupling equations that introduce interdependencies between disciplines, and a link function that eliminates the coupling variables and recovers the original mono-disciplinary problem. Theoretical conditions guaranteeing the equivalence between the two problems are established. The methodology is agnostic to the number of disciplines and variable dimensions, making it naturally suited for scalability studies. As an illustration, we construct a family of scalable MDO Rosenbrock problems and use them to benchmark two MDO coupling algorithms, namely the Jacobi and Gauss-Seidel schemes, across varying problem sizes. The proposed framework opens a systematic route to generating MDO benchmarks of arbitrary scale and complexity from the extensive catalog of existing mono-disciplinary test problems.



\end{abstract}

\section{Introduction}\label{sec:introduction}

Multidisciplinary design optimization (MDO) is a branch of optimization focusing on multidisciplinary problems~\cite{martins2021engineering}. Objective and constraint functions involve discipline-specific functions, and the difficulty lies in the fact that the latter are coupled: some of their inputs, called coupling variables, are outputs of other disciplines. This interdependence is captured through so-called coupling equations. Given the intense research activity devoted to designing efficient algorithms for solving MDO problems, and the need to compare their performance, there is a clear motivation for defining benchmark problems whose solutions are known theoretically.

The use of such benchmark problems is common practice in mono-disciplinary optimization. The literature is so abundant that several benchmarking platforms have been built around catalogs of optimization problems. The Comparing Continuous Optimizers platform (COCO) automates benchmarking experiments for black-box optimization algorithms by providing dozens of problems of various types, including unconstrained and constrained, mono-objective and bi-objective, continuous or mixed-integer, and noisy or noise-free \cite{Hansen02012021}. Similarly, CUTEst is a constrained and unconstrained testing environment for mathematical optimization, offering a collection of over 1000 problems \cite{gould2015cutest}.

By contrast, benchmark problems for multidisciplinary optimization remain scarce. Moreover, their solutions are rarely known theoretically and must instead be approximated numerically. Some of these problems are purely mathematical, with no physical interpretation, e.g., the Sellar problem \cite{sellar96}. Others rely on simple physical expressions, e.g., the combustion of propane problem \cite{tedfordmartins} and the supersonic business jet design problem \cite{sobieski-ssbj}, or address specific industrial applications \cite{abu2026establishing,volle2025aero,gray2024proposed}. Most of these problems have a fixed number of disciplines and fixed variable dimensions, which makes it difficult to assess whether conclusions drawn from a benchmark study remain robust at scale. Only a few scalable benchmark MDO problems address this limitation, either in a purely mathematical setting \cite{tedfordmartins,aziz2024scalable} or through a data-driven approach \cite{chauhan,vanaret2017}.

In this paper, we aim to bridge the gap between benchmark mono-disciplinary optimization problems and MDO. A first step in this direction was taken in \cite{johnson2023}, which extended the scalable ZDT bi-objective benchmark problem~\cite{ZDT} to the MDO setting. However, that approach is restricted to linear coupling functions and is tailored to a single specific optimization problem. In contrast, our work transforms any mono-disciplinary optimization problem with a known solution into a collection of parametric MDO problems sharing the same solution. This enables the construction of scalable MDO problems for which a solution is theoretically guaranteed. Our methodology relies on two ingredients: a set of coupling equations and a so-called link function. The link function eliminates the coupling variables, thereby reducing the MDO problem to a mono-disciplinary optimization problem. We provide concrete examples of both coupling equations and link functions to guide the user. The construction is underpinned by theoretical results that establish conditions ensuring the equivalence between the solution of the original mono-disciplinary problem and that of the resulting MDO problem. We illustrate the methodology by constructing scalable MDO Rosenbrock problems from their mono-disciplinary counterpart and use them to assess the performance of two MDO algorithms.

The paper is organized as follows. Section~\ref{sec:principle} introduces the background and notation. Section~\ref{sec:methodology} presents the methodology: Section~\ref{subsec:generic} describes the generic transformation of an optimization problem into an MDO problem, while Section~\ref{subsec:linear} develops a specialized version for linear coupling functions. Section~\ref{sec:experiments} applies the methodology to the multidimensional Rosenbrock problem to compare the Jacobi and Gauss-Seidel coupling algorithms across varying numbers of disciplines and variable sizes. Finally, Section~\ref{sec:conclusion} summarizes the methodology.

\section{Background and notations}\label{sec:principle}

In this work,
the optimization problems of interest take the form
\begin{equation}\label{eq:optpbm}    
\begin{aligned}
\min_{\bx\in\Xset}\quad & f(\bx) \\
\textrm{s.t.} \quad & g_0(\bx) \leq 0 \\
& g_i(\bx_0, \bx_i) \leq 0, \quad \iiniset. \\
\end{aligned}
\end{equation}
$\bx=(\bx_0, \bx_1, \ldots, \bx_N)$ are the optimization variables, 
$\Xset=\prod_{i=0}^N\Xset_i \subset\R^{d}$ is the search space with $\Xset_i\subset\R^{d_i}$,
$f: \defspace \to \R$ is the objective function, 
and $g_0:\defspace \to \R^{q_0}$ and $g_i:\defspace[0]\times\defspace[i] \to \R^{q_i}$ are the constraint functions.
This framework is sufficiently general 
to include a wide variety of optimization problems,
whether it is minimization or maximization, constrained or unconstrained problems,
with positivity or negativity constraints and arbitrary constraint threshold values.
Such problems can be rewritten as \eqref{eq:optpbm},
up to linear transformations and removal of constraints.
Note that the domain where $f, g_0, g_i$ are defined is the whole $\R^k$ (where $k=d, d_0, d_i$ respectively), a technical requirement that will be useful in the methodology. However, the optimization domain can be any subset of the definition domain.

Similarly,
the MDO problems of interest take the form
\begin{equation}\label{eq:mdopbm}    
\begin{aligned}
\min_{(\bx, \by)\in\XYset}\quad & \tilde{f}(\bx, \by) \\
\textrm{s.t.} \quad & \tilde{g}_0\left(\bx, \by \right)\leq 0 \\
& \tilde{g}_i\left(\bx_0, \bx_i,\by_{i}\right)\leq 0,\quad \iiniset \\
& \by_i = h_i(\bx_0, \bx_i, \by_{-i}), \quad \iiniset.
\end{aligned}
\end{equation}
$\by = (\by_1, \ldots, \by_N) \in \Yset \subset\R^p$ are the coupling variables, 
where $\by_i\in\Yset_i\subset\R^{p_i}$,
and $h_i: \Xset_0 \times \Xset_i \times \Yset_{-i} \to \Yset_i \subset\R^{p_i}$ is the $i$-th coupling function,
with $\by_{-i}=(\by_j)_{j\in \iset \setminus\{i\}}$.
A strongly coupled discipline is characterized by the specific features $\bx_i$, $\tilde{g}_i$ and $h_i$
and all the strongly coupled disciplines depend on the global optimization variables $\bx_0$.
In Problem \eqref{eq:mdopbm},
the coupling equation system $\{\by_i = h_i(\bx_0, \bx_i, \by_{-i}), \quad \iiniset\}$ models the interactions between $N\geq 2$ disciplines, which also imposes $d \geq 3$.
Lastly,
the quantities $\tilde{f}(\bx, \by)$ and $\tilde{g}_0\left(\bx, \by \right)$ are computed by 
either some of these strongly coupled disciplines or complementary disciplines,
which are coupled or not with the latter.

Note that Problem \eqref{eq:mdopbm} is sufficiently generic to be rewritten according to the MDF and IDF architectures~\cite{martinsLambe}, which are the most popular MDO formulations. It can also be broken down into sub-problems, using bi-level formulations~\cite{gazaix2019}.

\section{Methodology to make an optimization problem multidisciplinary}\label{sec:methodology}

This section presents a methodology to rewrite the optimization problem \eqref{eq:optpbm} as the MDO problem \eqref{eq:mdopbm}.
This technique has its origins in the PhD thesis of one of the authors \cite{azizalaoui:tel-05059696}.
The idea is to consider a set of coupling functions $\{h_1, \dots, h_N \}$
and to define a \textit{link} function $L:\XYset \to \R^d$
such that $L(\bx,\by) = \bx$ for all solutions $(\bx,\by)$ of the coupling equations $\{\by_i = h_i(\bx_0, \bx_i, \by_{-i}), \quad \iiniset\}$. Indeed, we can see that $L$ eliminates the coupling variables $\by$ and thus allows to go from an MDO problem to an optimization one.
The connection between Problems \eqref{eq:optpbm} and \eqref{eq:mdopbm} is then made by composing the objective and constraint functions $f,g_0,\ldots,g_N$ with $L$.

\subsection{A generic transformation}\label{subsec:generic}

We start from the framework of Problem \eqref{eq:optpbm}.
We further consider extra functions $h_i: \Xset_0 \times \Xset_i \times \Yset_{-i} \to \Yset_i$  ($\iiniset$), corresponding to the coupling functions of the MDO problem that we wish to create. For simplicity, denote 
$g: \bx \in \defspace \mapsto (g_0(\bx), g_1(\bx_0, \bx_1), \dots, g_N(\bx_0, \bx_N))$ and 
$h: (\bx, \by) \in \XYset \mapsto (h_i(\bx_0, \bx_i, \by_{-i}))_{i\in\iset} \in \Yset$.

\begin{proposition}
\label{prop:main_result}
Let $\Sset=\{(\bx,\by)\in\XYset \mid y=h(\bx, \by)\}$ be the set of solutions of the coupling equations. Let $\SprojX = \{\bx \in \Xset \mid \exists \by \in \Yset,~(\bx, \by) \in \Sset \}$ be the projection of $\Sset$ onto $\Xset$.\\
Let $L:\XYset\to\R^d$ be a function such that  $L(\bx,\by)=\bx$ for all $(\bx, \by) \in \Sset$. Define $\tilde{f}=f\circ L$ and $\tilde{g}=g\circ L$ over $\XYset$.\\
Then solving Problem \eqref{eq:mdopbm} on $\XYset$ is equivalent to solving Problem \eqref{eq:optpbm} on $\SprojX$,
i.e. $(\bx, \by)$ is a solution of \eqref{eq:mdopbm} on $\XYset$ 
if and only if $\bx$ is a solution of \eqref{eq:optpbm} on $\SprojX$.
Furthermore, their solutions do not depend on $L$.
\end{proposition}

\begin{proof}
Problem \eqref{eq:mdopbm} over $\XYset$ can be rewritten over $\Sset$ as
\begin{equation*}
\begin{aligned}
\min_{(\bx, \by)\in \Sset}\quad & \tilde{f}(\bx, \by) \\
\textrm{s.t.} \quad & \tilde{g}_0\left(\bx, \by \right)\leq 0 \\
& \tilde{g}_i\left(\bx_0, \bx_i,\by_{i}\right)\leq 0,\quad \iiniset
\end{aligned}
\end{equation*}
where the coupling equations $h(\bx,\by)=\by$ have been moved from the constraints to the optimization domain $\Sset$.
Furthermore, by the definition of $L$, $\tilde{f}$ and $\tilde{g}$, we have: $\forall (\bx, \by) \in \Sset, \tilde{f}(\bx,\by)=f(\bx)$ and $\tilde{g}(\bx,\by)=g(\bx)$.
Then, Problem \eqref{eq:mdopbm} is rewritten
\begin{equation*}
\begin{aligned}
\min_{(\bx, \by)\in \Sset}\quad & f(\bx) \\
\textrm{s.t.} \quad & g_0(\bx)\leq 0 \\
& g_i(\bx_0, \bx_i)\leq 0,\quad \iiniset
\end{aligned}
\end{equation*}
Now, this can be viewed as a standard optimization problem with respect to $\bx$ since $\by$ takes place only in the optimization domain. More precisely, $\bx$ must be such that for some $\by$ we have $(\bx,\by) \in \Sset$, i.e. $\bx \in \SprojX$. Thus, solving Problem \eqref{eq:mdopbm} on $\XYset$ is equivalent to solving Problem \eqref{eq:optpbm} on $\SprojX$.\\
Finally, the solutions of these equivalent problems depend only on $\SprojX$, which depends on $h$ but not on $L$.
\end{proof}

A preliminary remark is that there is a wide choice for the link function $L$.
A simple example is given by $L(\bx,\by) = (\bx_0, \bx_{-0}+h(\bx,\by)-\by)$.
More generally, $L$ can be any function of the form
$ L(\bx,\by) = T(\bx, h(\bx,\by) - \by)$
where $T : \Xset \times \R^p \to \R^d $ is such that $T(\bx, 0) = \bx$ for all $\bx \in \Xset$. For instance, with $T(\bx,\bz) = \bx \exp(a \Vert \bz \Vert)$, for some parameter $a\in\R$, we obtain
$L(\bx,\by) = \bx \exp(a \Vert h(\bx,\by) - \by \Vert)$
where $\|\cdot\|$ is the Euclidean norm.

Besides, the methodology as a whole is generic. 
Indeed, it allows to create an MDO problem from any optimization problem, by combining any coupling equations $h$ and any link function $L$.
This framework is also scalable, i.e. adaptable to reference optimization problems whose search and constraint space dimensions are configurable. Moreover, it makes no assumptions about the coupling functions $h$, the search space $\Xset$ and the coupling space $\Yset$. 
It also handles cases where the solution of the coupling equations is not unique
since $L(\bx,\by)=\bx$ for every solution $\by$.

Finally, it is important to note that this methodology guarantees the equivalence between Problem \eqref{eq:optpbm} and \eqref{eq:mdopbm} only on a subset $\SprojX$ of the reference search domain $\Xset$, i.e. where the coupling equations admit a solution. 
This can pose a difficulty, as we shall see later.
However, Proposition \ref{prop:nonlinear_result} inspired by \cite{vanaret2017} states that it suffices to use coupling functions $h$ that are continuous and bounded, so that the equivalence is over the entire search domain $\Xset$.

\begin{proposition}
\label{prop:nonlinear_result}
Let $\Yset$ be a convex compact set. Let the coupling functions $h: (\bx, \by) \in \XYset \mapsto h(\bx,\by)\in\Yset$ be continuous according to $\by$.
Let $L:\XYset\to\R^d$ be a function such that $L(\bx,\by)=\bx$ for all $(\bx, \by) \in \Sset$. Define $\tilde{f}=f\circ L$ and $\tilde{g}=g\circ L$ over $\XYset$.
Then solving Problem \eqref{eq:mdopbm} on $\XYset$ is equivalent to solving Problem \eqref{eq:optpbm} on $\Xset$.
Furthermore, their solutions do not depend on $L$.
\end{proposition}

\begin{proof}
Let $\bx\in\Xset$ and $\varphi_{\bx}:\by\in\Yset\mapsto h(\bx,\by) \in\Yset$ a continuous function where $\Yset$ is a convex compact set. 
There exists a $\by\in\Yset$ such that $\varphi_{\bx}(\by)=\by$ according to the Brouwer's fixed point theorem~\cite{brouwer1911abbildung,park1999ninety}.
Consequently, there exists a solution to the coupling equations for any $\bx\in\Xset$, i.e. $\SprojX=\Xset$.
Therefore, by definition of $L$, $\tilde{f}$ and $\tilde{g}$, Proposition \ref{prop:main_result} states that Problem \eqref{eq:mdopbm}  is equivalent to Problem \eqref{eq:optpbm} on $\Xset$ and their solutions do not depend on $L$.
\end{proof}

Mind that Proposition \ref{prop:nonlinear_result} requires, in addition to continuity, that the coupling functions $h$ have its values in $\Yset$, which allows to apply Brouwer's theorem.
A simple example is given by
\begin{equation}\label{eq:nl1}
h_i(\bx_0, \bx_i, \by_{-i}) = \left(m_{i,j}+(M_{i,j}-m_{i,j})\sigma_{i,j}\!\left(\psi_{i,j}(\bx_0,\bx_i,\by_{-i})\right)\right)_{j\in\llbracket 1,p_i \rrbracket}
\end{equation}
where $\psi_{i,j}$ is the $j$-th component of $\psi_i:\Xset_0\times\Xset_i\times\Yset_{-i}\to\Yset_i$ and $\sigma_{i,j}:\R\to[0,1]$ are continuous functions, with $\Yset=\prod_{i=1}^N\prod_{j=1}^{p_i}[m_{i,j},M_{i,j}]$.
For instance, 
\begin{equation}\label{eq:nl2}
\psi_i:(\bx_0,\bx_i,\by_{-i})\mapsto\ba_i - \bB_{i,0} \bx_0 - \bB_{i,i} \bx_i + \sum_{j \neq i} \bC_{i,j} \by_j
\end{equation}
with $\bB_{i,i}\in \mat{p_i}{d_i}$, $\bC_{i,j}\in\mathcal{M}_{p_i}(\R)$ and $\ba_i\in\mathcal{M}_{p_i,1}(\R)$.
Classical examples of $\sigma$ are given by the sigmoid functions used in the artificial neural networks, such as the logistic function $\sigma(x)=(1+e^{-ax})^{-1}$ with $a>0$.
Note that this choice of $\sigma$ makes the coupling functions $h$ non-linear, which may be of practical interest when benchmarking algorithms.\\

When $h$ and $\Yset$ are not as defined in Proposition \ref{prop:nonlinear_result}, the application of the methodology may be less straightforward. To illustrate this difficulty, let consider the three-dimensional Rosenbrock problem \cite{rosenbrock1960} over $\Xset =[-2,2]^3$, with known solution $\bx^*=(1,1,1)$, and two disciplines with the following coupling functions:
$$ \begin{cases}
&h_1(x_0, x_1, y_2) = b_1(x_0, x_1) + \sgn(x_1) y_2 \\
&h_2(x_0, x_2, y_1) = b_2(x_0, x_2) + \sgn(x_2) y_1 
\end{cases} $$
where $b_1$ and $b_2$ are some functions, and $\sgn(x)$ is the sign function (with the convention $\sgn(0) = 0$). The functions $b_1$ and $b_2$ can be very general, but we require that $b_1(x_0, x_1) \neq - b_2(x_0, x_2)$ when $x_1$ and $x_2$ have the same sign; for instance, we can choose two (strictly) positive functions.
The coupling equations $h_i(x_0, x_i, y_{-i})=y_i$ ($i=1, 2$) form the  linear system
$$ A(x) \begin{pmatrix}
    y_1 \\ y_2
\end{pmatrix} = \begin{pmatrix}
    b_1(x_0, x_1) \\ b_2(x_0, x_2)
\end{pmatrix}$$
where $A(x) = \begin{pmatrix}
    1 & -\sgn(x_2) \\ -\sgn(x_1)&  1
\end{pmatrix}$. 
We can see that these functions are not continuous and that $(\bx,\by) \in \Sset$ if and only if $\det(A(\bx)) \neq 0$. Indeed, if $\det(A(x))=0$, the  system does not have a solution due to the requirements on $b_1, b_2$. Thus, $(\bx,\by) \in \Sset$ if and only if $x_1$ and $x_2$ do not have the same sign, and 
$$\SprojX = ([-2,2] \times [0,2] \times [-2, 0]) \cup ([-2,2] \times [-2,0] \times [0, 2])$$
which is a strict subset of $\Xset$, representing half of the reference search space as illustrated in Figure \ref{fig:SprojX}. 
Thus, solving the MDO problem is equivalent to solving the original optimization problem on a smaller part $\SprojX$ of the search domain $\Xset$. 
More troublesome, $\SprojX$ does not contain the reference solution $\bx^*$ and the solution of the original optimization problem over $\SprojX$ may be unknown. This makes it impossible to use the MDO problem for benchmarking algorithms. In this case, the difficulty can easily be overcome by modifying the coupling functions as
$$ \begin{cases}
&h_1(x_0, x_1, y_2) = b_1(x_0, x_1) + \sgn(x_1) y_2 \\
&h_2(x_0, x_2, y_1) = b_2(x_0, x_2) - \sgn(x_2) y_1 
\end{cases} $$
where $b_1(x_0, x_1) \neq - b_2(x_0, x_2)$ when $x_1$ and $x_2$ do not have the same sign. $\SprojX$ is then written as
$$\SprojX = ([-2,2] \times [0,2] \times [0, 2])$$
and $\bx^*\in\SprojX$.
Thus, the reference solution can be used to compare the performance of MDO algorithms.

\begin{figure}[!ht]
\centering
\begin{tikzpicture}[line join=round, line cap=round, thick, scale=1.2]

\pt{P000}{-1}{-1}{-1}
\pt{P001}{-1}{-1}{ 1}
\pt{P010}{-1}{ 1}{-1}
\pt{P011}{-1}{ 1}{ 1}
\pt{P100}{ 1}{-1}{-1}
\pt{P101}{ 1}{-1}{ 1}
\pt{P110}{ 1}{ 1}{-1}
\pt{P111}{ 1}{ 1}{ 1}

\pt{A_m10m1}{-1}{0}{-1}
\pt{A_m100}{-1}{0}{0}
\pt{A_m11m1}{-1}{1}{-1}
\pt{A_m110}{-1}{1}{0}
\pt{A100m1}{ 1}{0}{-1}
\pt{A1000}{ 1}{0}{0}
\pt{A101m1}{ 1}{1}{-1}
\pt{A1010}{ 1}{1}{0}

\pt{B_m1m10}{-1}{-1}{0}
\pt{B_m1m11}{-1}{-1}{1}
\pt{B_m100}{-1}{0}{0}
\pt{B_m101}{-1}{0}{1}
\pt{B10m10}{ 1}{-1}{0}
\pt{B10m11}{ 1}{-1}{1}
\pt{B1000}{ 1}{0}{0}
\pt{B1001}{ 1}{0}{1}

\fill[black!45] (A1000)--(A1010)--(A101m1)--(A100m1)--cycle;
\fill[black!30]
  (A_m100)--(A1000)--(A1010)--(A_m110)--cycle;
\fill[black!20]
  (A_m11m1)--(A_m110)--(A1010)--(A101m1)--cycle;

\fill[black!45] (B1001)--(B10m11)--(B10m10)--(B1000)--cycle;
\fill[black!30]
  (B_m1m11)--(B10m11)--(B1001)--(B_m101)--cycle;
\fill[black!20]
  (B_m100)--(B1000)--(B1001)--(B_m101)--cycle;

\draw (P100)--(P101);
\draw (P110)--(P111);
\draw (P010)--(P011);
\draw (P100)--(P110)--(P010);
\draw (P001)--(P101)--(P111)--(P011)--cycle;

\coordinate (O) at (0,0);
\draw[->,dashed] (O) -- (2,0) node[right] {$x_0$};
\pt{A}{0.5}{0}{0}
\pt{B}{0.5}{0}{0.5}
\draw[dotted] (A) -- (B);
\pt{A}{0}{0}{0.5}
\pt{B}{0.5}{0}{0.5}
\draw[dotted] (A) -- (B);
\pt{A}{0.5}{0}{0.5}
\pt{B}{0.5}{0.5}{0.5}
\draw[dotted] (A) -- (B);
\pt{A}{0}{0.5}{0}
\pt{B}{0}{0.5}{0.5}
\draw[dotted] (A) -- (B);
\pt{A}{0}{0.5}{0.5}
\pt{B}{0}{0}{0.5}
\draw[dotted] (A) -- (B);
\pt{A}{0}{0.5}{0.5}
\pt{B}{0.5}{0.5}{0.5}
\draw[dotted] (A) -- (B);
\pt{A}{0.5}{0.5}{0.5}
\pt{B}{0.5}{0.5}{0}
\draw[dotted] (A) -- (B);
\pt{A}{0.5}{0.5}{0}
\pt{B}{0.5}{0}{0}
\draw[dotted] (A) -- (B);
\pt{A}{0}{0.5}{0}
\pt{B}{0.5}{0.5}{0}
\draw[dotted] (A) -- (B);
\pt{M}{0.5}{0.5}{0.5}
\draw[fill=white] (M) circle (2pt);
\draw[->,dashed] (O) -- (0,2) node[above] {$x_2$};
\draw[->,dashed] (O) -- (-1.2,-0.8) node[below left] {$x_1$};

\end{tikzpicture}

\caption{Illustration of the domain issue when the assumptions of Proposition~\ref{prop:nonlinear_result} are not satisfied. The equivalent search space $\SprojX$ (in gray) is a strict subset of the reference search space $\Xset=[-2,2]^3$ of the three-dimensional Rosenbrock problem when the coupling functions are $h_i(x_0, x_i, y_{-i}) = b_i(x_0, x_i) + \sgn(x_i) y_{-i}$ ($i=1,2$). The solution $\bx^*=(1,1,1)$ of the reference optimization problem represented by a white dot is outside this domain.}
\label{fig:SprojX}

\end{figure}
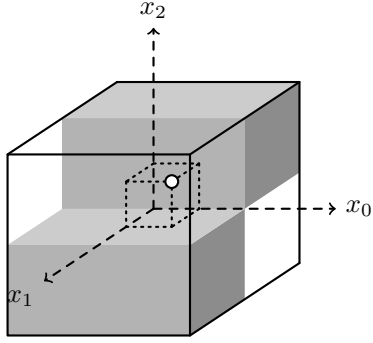

\subsection{Case of linear coupling functions with constant coefficients}\label{subsec:linear}

In this section, we propose a case where the coupling functions are linear with constant coefficients sampled independently at random, so that the coupling equations admit an explicit solution.
This can be useful to benchmark the coupling algorithms involved in the MDO formulations such as Jacobi and Gauss-Seidel technique in the case of the MDF formulation.

\begin{proposition}\label{prop:linear}
Let us consider Problem \eqref{eq:mdopbm} with the linear coupling functions 
\begin{equation}\label{eq:linear_h}
\iiniset,\quad 
h_i(\bx_0, \bx_i, \by_{-i}) = \ba_i - \bB_{i,0} \bx_0 - \bB_{i,i} \bx_i + \sum_{j \neq i} \bC_{i,j} \by_j, 
\end{equation}
where $\bB_{i,i}\in \mat{p_i}{d_i}$, $\bC_{i,j}\in\mathcal{M}_{p_i}(\R)$ and $\ba_i\in\mathcal{M}_{p_i,1}(\R)$.
This multidisciplinary system can be expressed in a matrix form as
$$\bC\by = \ba - \bB \bx$$ 
with 
$$
\ba =
\begin{pmatrix}\ba_1\\\vdots\\\ba_N\end{pmatrix}
\bB = 
\begin{pmatrix}
\bB_{i,0} & \bB_{1,1} & \cdots & \bzero \\
\vdots  & \vdots  & \ddots & \vdots  \\
\bB_{N,0} & \bzero & \cdots & \bB_{N,N}
\end{pmatrix}
\quad
\bC = 
\mathbf{I}-\begin{pmatrix}
\bzero & \bC_{1,2} & \cdots & \bC_{1,N} \\
\bC_{2,1} & \bzero_{p_2} & \ddots & \bC_{2,N} \\
\vdots & \ddots & \ddots & \vdots \\
\bC_{N,1} & \bC_{N,2} & \ldots & \bzero \\
\end{pmatrix}
$$
where $\mathbf{I}$ and $\bzero$ denote the identity and zeros matrices respectively, with flexible shape. Assume that $\bC$ is invertible.
Let $L:\XYset\to\R^d$ be a function such that  $L(\bx,\by)=\bx$ for all $(\bx, \by) \in \Sset$. Define $\tilde{f}=f\circ L$ and $\tilde{g}=g\circ L$ over $\XYset$.\\
Then, Problem \eqref{eq:mdopbm}  is equivalent to Problem \eqref{eq:optpbm} on $\SprojX = \{\bx \in \Xset \mid \bC^{-1}(\ba - \bB \bx) \in \Yset \}$.

An example of explicit function $L$ is given by
\begin{equation}\label{eq:linear_link}
L(\bx,\by)=\left(\bx_0,\bx_1 + \by_1 - \balpha_1 - \sum_{j=0}^N \bbeta_{1, j}\bx_j,\ldots,\bx_N + \by_N - \balpha_N - \sum_{j=0}^N \bbeta_{N, j}\bx_j\right)
\end{equation}
where 
$
\bbeta = \begin{pmatrix}
\bbeta_{1,0} & \bbeta_{1,1} & \cdots & \bbeta_{1,N} \\
\vdots  & \vdots  & \ddots & \vdots  \\
\bbeta_{N,0} & \bbeta_{N,1} & \cdots & \bbeta_{N,N} 
\end{pmatrix}=-\bC^{-1} \bB
$
and
$ \balpha = \begin{pmatrix}
\balpha_{1} \\
\dots \\
\balpha_{N}    
\end{pmatrix} = \bC^{-1} \ba.
$
\end{proposition}

\begin{proof}
The proof mimics the one of Proposition \ref{prop:main_result}, with some simplifications.
As $C$ is assumed invertible, the coupling equation $h(\bx, \by) = \by$ is equivalent to the explicit form 
$$
\by = c(\bx) = -\bC^{-1}\bB \bx + \bC^{-1}\ba = \boldsymbol{\beta} \bx + \boldsymbol{\alpha}
$$
Then, consider the function $L$ defined by
\[
L(\bx, \by) = (\bx_0, \bx_1 + \by_1 - c_1(\bx), \ldots, \bx_N + \by_N - c_N(\bx)).
\]
Clearly,
the condition $L(\bx, \by) = \bx$ is equivalent to $c(\bx) = \by$, which is in turn equivalent to $h(\bx, \by) = \by.$
The rest of the proof is then exactly the same as in the proof of Proposition \ref{prop:main_result}. We do not reproduce it for the sake of brevity. Notice that, here, $\SprojX$ has an explicit form as finding $\bx \in \Xset$ such that $h(\bx, \by) = \by$ is equivalent to find $\bx \in \Xset$ such that $\bC^{-1}(\ba - \bB \bx) \in \Yset$.
\end{proof}

In practice, we may choose $\Yset=\R^p$, which leads to $\SprojX=\Xset$. This makes this configuration usable as is.
Otherwise, for a general choice of $\Yset$,
the optimization domain $\SprojX=\{\bx \in \Xset \mid \bC^{-1}(\ba - \bB \bx) \in \Yset \}$ is harder to compute. 
Nevertheless, it may be used in some cases. For instance, if $\Yset=[m,M]^p$, then $x\in\SprojX$ is equivalent to the linear inequality constraints $m\leq (\bC^{-1}(\ba - \bB \bx))_i \leq M$, $\forall i\in\llbracket 1,p\rrbracket$.

Moreover, the condition that $\bC$ is invertible is easy to obtain. Indeed, it is verified with probability 1 when its coefficients are sampled independently at random, as proved in Proposition~\ref{prop:Cinvertible}. The latter result is not surprising as unstructured random matrices are known to be invertible with probability one (see, e.g., \cite{zeng2024expressive}). However, $\bC$ has here a specific structure, involving deterministic and random block matrices. Thus, we provide a dedicated proof. 

\begin{proposition}\label{prop:Cinvertible}
	Let $\mu$ a continuous probability measure on $\R$, i.e. admitting a density function with respect to the Lebesgue measure. Let us consider the matrix $\bC$ defined in Proposition \ref{prop:linear}.
	If the coefficients of $\bC$ are sampled independently at random from $\mu$, then $\bC$ is invertible with probability $1$.
\end{proposition}
The proof relies on the following lemma.
\begin{lemma} \label{lemma:Cinvertible}
	Let $\mu$ be a continuous probability measure on $\R$, and let $U_1, \dots, U_n$ be independent random variables with probability distribution $\mu$.\\
	Let $P(x_1, \dots, x_n)$ be a non-zero polynomial with degree at most $1$ with respect to each $x_i$ ($i= 1, \dots, n$).
	Then, $\prob(P(U_1, \dots, U_n) = 0) = 0$.
\end{lemma}
To see how Proposition~\ref{prop:Cinvertible} results from Lemma~\ref{lemma:Cinvertible}, let us recall the definition of the determinant of $\bC$, 
$$ \det(\bC) = \sum_{\sigma \in S_n} \epsilon(\sigma) \prod_{i=1}^n [\bC]_{i, \sigma(i)}.$$
Here $S_n$ is the set of all bijections of $\{1, \dots, n\}$, $\varepsilon(\sigma)$ is the signature of $\sigma$ (which can be equal to $1$ or $-1$), and $[\bC]_{i,j}$ is the coefficient $(i,j)$ of $\bC$. As $\sigma$ are bijections, $\det(\bC)$ is a polynomial of the non-constant coefficients of $\bC$, of degree at most $1$ with respect to each of them. 
It is non-zero because its constant term
(obtained by setting all non-constant coefficients of $\bC$ to zero), is equal to $1$. 
Thus, if the coefficients of $\bC$ are sampled independently at random from $\mu$, then by Lemma~\ref{lemma:Cinvertible}, 
$$ \prob(\det(\bC) = 0) = 0,$$
meaning that $\bC$ is invertible with probability $1$.\\
It remains to prove Lemma~\ref{lemma:Cinvertible}.
\begin{proof} of Lemma~\ref{lemma:Cinvertible}.\\
	The proof is by induction on $n$.\\
	If $n=1$, we have $P(x_1) = a x_1 + b$ where $a, b$ are not all zero. If $a=0$, then $b \neq 0$ and $\prob(P(U_1) = 0) = \prob(b =0) = 0$. If $a \neq 0$, then $\prob(P(U_1) = 0) = \prob(U_1 = -b/a)$ which is equal to zero as $\mu$ is a continuous probability measure.\\
	Now, let us assume that the lemma is true for all integers up to $n$, and let us prove it for $n+1$. Thus let $P_{n+1}(x_1, \dots, x_{n+1})$ be a non-zero polynomial whose degree at most $1$ with respect to each $x_i$ ($i=1, \dots, n+1$). 
	Thus, we can write,
	$$ P_{n+1}(x_1, \dots, x_{n+1}) = x_{n+1} Q(x_1, \dots, x_n) + R(x_1, \dots, x_n) $$
	where $Q, R$ are polynomials whose degree at most $1$ with respect to their variables, and where at least one of $Q, R$ is non-zero. \\
	If $Q \equiv 0$, then $P_{n+1}(x_1, \dots, x_{n+1}) = R(x_1, \dots, x_n)$, where $R$ is non-zero. Thus, by the induction property applied to $R$, we get $\prob(P_{n+1}(U_1, \dots, U_{n+1})=0) = 0$.\\
	Now, let consider the other case, i.e. $Q$ is non-zero. Then, by the induction property applied to $Q$, $\prob(Q(U_1, \dots, U_{n}) \neq 0) = 1$. Thus, we have
	\begin{eqnarray*}
	& \prob(P_{n+1}(U_1, \dots, U_{n+1}) = 0) = \\
	& \int_{Q(x_1, \dots, x_n) \neq 0} \prob(U_{n+1} Q(x_1, \dots, x_n) + R(x_1, \dots, x_n) = 0) d\mu_1(x_1) \dots d\mu_n(x_n) \\
	\end{eqnarray*}
	The integral is equal to zero, since when $Q(x_1, \dots, x_n) \neq 0$, the probability that $U_{n+1}$ equals $-R(x_1, \dots, x_n)/\allowbreak Q(x_1, \dots, x_n)$ is equal to zero as $\mu$ is a continuous probability measure. This concludes the proof.
\end{proof}

\section{Numerical experiments}\label{sec:experiments}

\subsection{Reference optimization problem}

This section illustrates the methodology on an optimization problem 
using the multidimensional Rosenbrock function \cite{rosenbrock1960,rosenbrock1975}:
\begin{equation}\label{eq:rosenbrock}
    \min_{\bx\in[-2,2]^N} f(\bx)
\end{equation}
where
$$f(\bx)=\sum_{i=1}^{N-1} 100(x_{i+1}-x_i^2)^2 + (1-x_i)^2$$

This minimization problem admits a unique solution $\bx^*$,
namely the unit vector $(1,1,\ldots)$ of $\R^N$ where $f(\bx^*)=0$;
it will be used as the reference.
For $N\geq 4$,
the function $f: \R^N \to \R_+$ has also a local minimum equal to 4 at point $\tilde{\bx}=(-1,1,1,\ldots)$ \cite{Shang2006}.
In addition to this local minimum $\tilde{\bx}$, 
many saddle points can make the search of the global optimum $\bx^*$ a hard task~\cite{Rosenbrock2009}.

\subsection{Objective}

The objective of this numerical study is to compare the performance of two coupling algorithms
for solving an MDO version of Problem \eqref{eq:rosenbrock} using the MDF formulation.
These algorithms are the Jacobi and Gauss-Seidel point-fixed methods.
All other things are equal,
in particular the gradient-based optimization algorithm SLSQP.

The metrics of interest are 
the Euclidean distance $\Delta_{\bx} = \|\bx-\bx^*\|_2$ to $\bx^*$,
the distance $\Delta_f = |f(\bx)-f(\bx^*)|$ to $f(\bx^*)$ 
and the numbers of evaluations $n_{\phi}$ and differentiations $n_{\phi}'$ of the different functions $\phi\in\{f,L,h_1,\ldots,h_N\}$. 
We will also consider $\bar{n}_h$ and $\bar{n}_h'$ defined as the averages of $n_{h_1},\ldots,n_{h_N}$ and $n_{h_1}',\ldots,n_{h_N}'$ respectively.
The comparison will be repeated 100 times using different initial values of $\bx$ obtained from a Latin hypercube sampling (LHS), and the results will be displayed as \texttt{mean (standard deviation)}.

The open-source Python library GEMSEO\textsuperscript{\tiny\textsf{\textregistered}} (https://www.gemseo.org) dedicated to MDO is used to carry out this study \cite{gemseo2018},
with the NLopt implementation \cite{NLopt} of the SLSQP algorithm \cite{SLSQP} and the OpenTURNS library \cite{openturns} for design of experiments.

For the sake of clarity,
we start with the 3-dimensional Rosenbrock problem (i.e., $\bx=(x_0,x_1,x_2)$ and $N=2$),
whose associated MDO problem involves two strongly coupled disciplines. 
More precisely,
Section \ref{subsec:construction} will illustrate the transformation of Problem \eqref{eq:rosenbrock} into an MDO problem
and Section \ref{subsec:problem_1} will compare the coupling algorithms using different starting points.
Then,
Sections \ref{subsec:problem_2} and \ref{subsec:problem_3} will carry out the same type of comparison,
using six disciplines with scalar design variables and two disciplines with vectorial design variables respectively.
Finally, Section \ref{subsec:problem_4} will consider the case of non-linear couplings.

\subsection{Construction of the MDO problem with scalar design variables}\label{subsec:construction}

By Proposition \ref{prop:linear},
the optimization problem \eqref{eq:rosenbrock} can be transformed into the MDO problem
\begin{equation}\label{eq:rosenbrockmdo}
\begin{aligned}
\min_{\bx\in[-2,2]^3\atop \by\in\R^3}\quad & \tilde{f}(\bx, \by) \\
\textrm{s.t.} \quad & y_1 = h_1(x_0, x_1, y_2) \\
& y_2 = h_2(x_0, x_2, y_1).
\end{aligned}
\end{equation}
The latter can be reformulated using the MDF formulation as
\begin{equation}\label{eq:rosenbrockmdf}
\min_{\bx\in[-2,2]^3}~\tilde{f}(\bx, \by(\bx))
\end{equation}
where $\by(\bx)=\bC^{-1}(\ba-\bB\bx)$ is the solution of the coupling equation system at $\bx$.

Figure \ref{fig:coupling_graph} shows the coupling graph of the resulting multidisciplinary system.
The design variables are represented by boxes,
the coupling variables by diamonds
and the disciplines (equivalent to functions in this graph) by circles.
The disciplines $h_1$ and $h_2$ are strongly coupled, 
as indicated by the thickness of the line,
$L$ is the discipline defined by the link function
and $f$ is the original 3-dimensional Rosenbrock function.

\begin{figure}[!h]
    \centering
    \includegraphics[width=\linewidth]{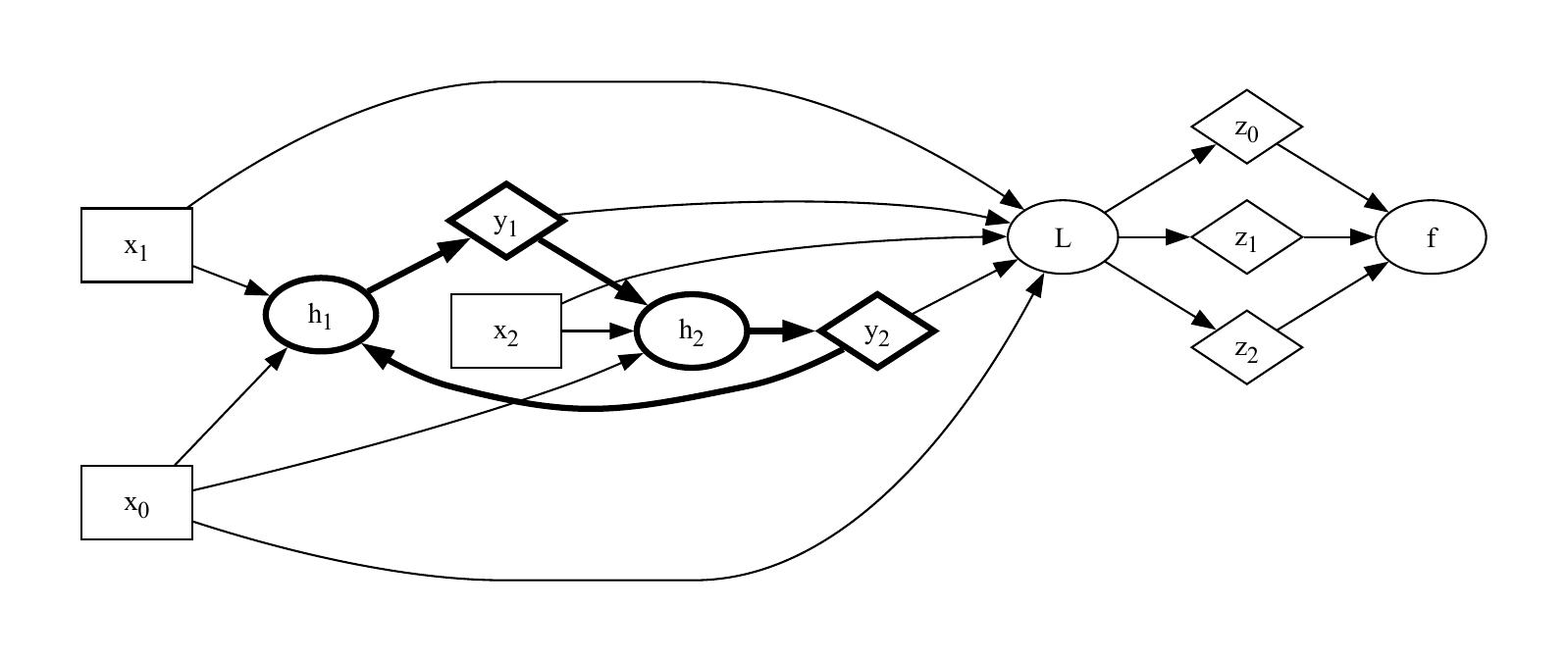}
    \caption{Coupling graph of the 3-dimensional Rosenbrock MDO problem.}
    \label{fig:coupling_graph}
\end{figure}

In practice, the Jacobi or Gauss-Seidel technique approximates $\by(\bx)$ by $\hat{\by}(\bx)$ and consequently, shifts the objective by $\tilde{f}(\bx, \hat{\by}(\bx))-\tilde{f}(\bx, \by(\bx))$:
\begin{equation}\label{eq:rosenbrockmdfapprox}
\min_{\bx\in[-2,2]^3}~\tilde{f}(\bx, \hat{\by}(\bx)).
\end{equation}
The worse the coupling problem is solved, the further the MDO problem deviates from the reference optimization problem.
Figure \ref{fig:convergence} shows the convergence of the SLSQP optimizer from the starting point $\bx^{(0)}=(0.29,0.95,0.97)$ to the theoretical solution, when using the Jacobi algorithm.
The optimizer stopped after 30 iterations, 
with a plateau reached around the ninth iteration for the objective and twenty-fifth for the design variables. 
We can see the gain in precision of the latest iterations, 
of nearly a decade per iteration.

\begin{figure*}[!h]
    \centering
    \includegraphics[width=0.49\linewidth]{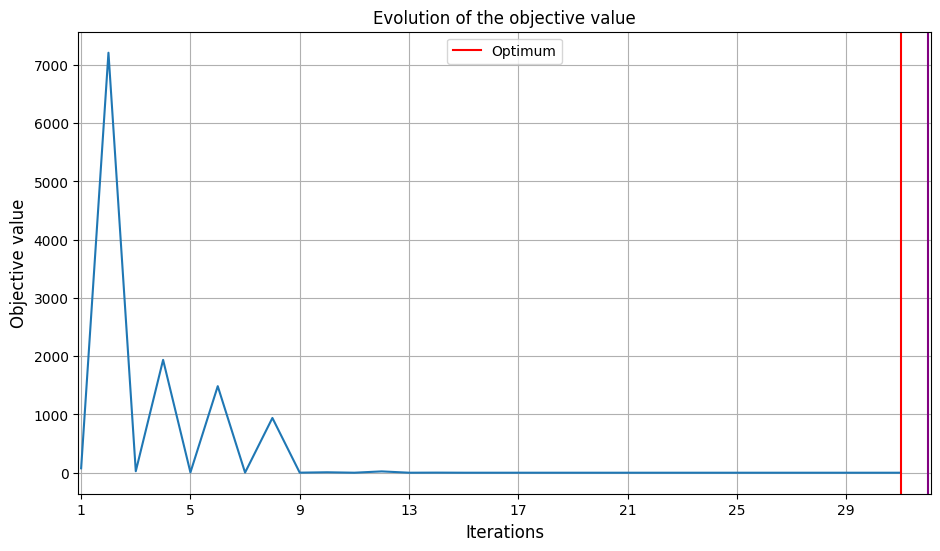}
    \includegraphics[width=0.44\linewidth]{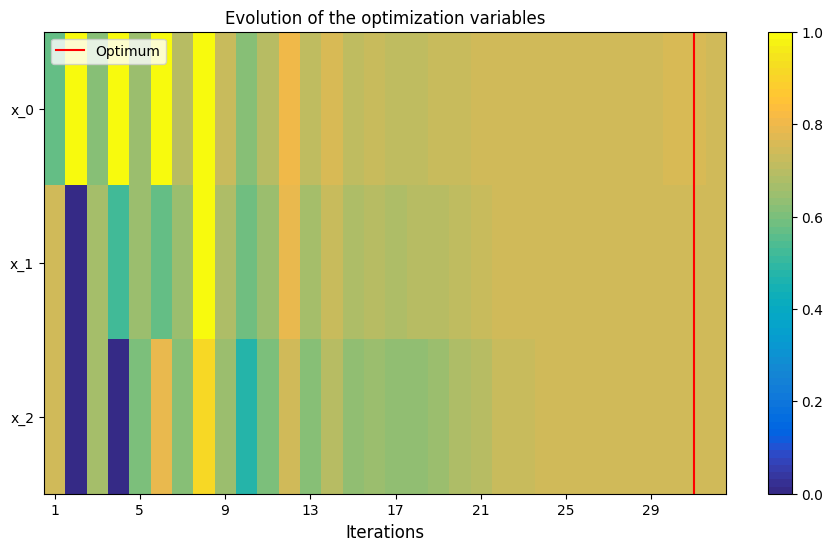}
    \includegraphics[width=0.49\linewidth]{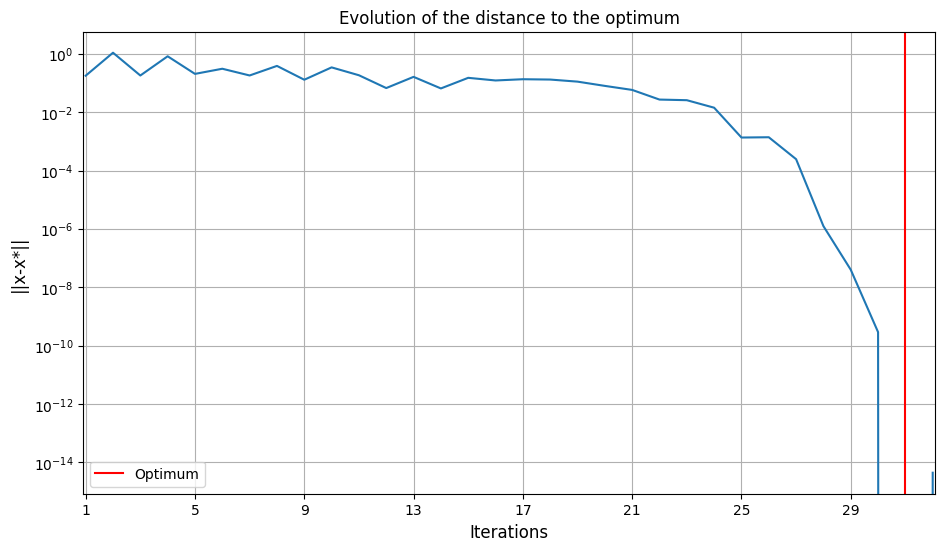}
    \caption{Convergence of the Jacobi algorithm from the starting point $\bx^{(0)}=(0.29,0.95,0.97)$, in terms of objective, optimization variables and logarithm of the Euclidean distance $\Delta_{\bx}$ to the optimum $\bx^*$.}
    \label{fig:convergence}
\end{figure*}

The next section will repeat this experiment for both Gauss-Seidel and Jacobi algorithms from different starting points $\bx^{(0)}$,
in order to see whether the parallel structure of Jacobi is an advantage over Gauss-Seidel or whether the numerical stability of Gauss-Seidel makes its superiority.
Before that,
this section concludes by Figure \ref{fig:xdsm} comparing the XDSM (eXtended Design Structure Matrix) \cite{xdsm} of the MDF formulation using the Gauss-Seidel algorithm (Figure \ref{fig:xdsmgs}) and the XDSM of the MDF formulation using the Jacobi one (Figure \ref{fig:xdsmj}).

\begin{figure*}[!h]
    \centering
     \begin{subfigure}[b]{0.75\textwidth}
        \includegraphics[width=\linewidth]{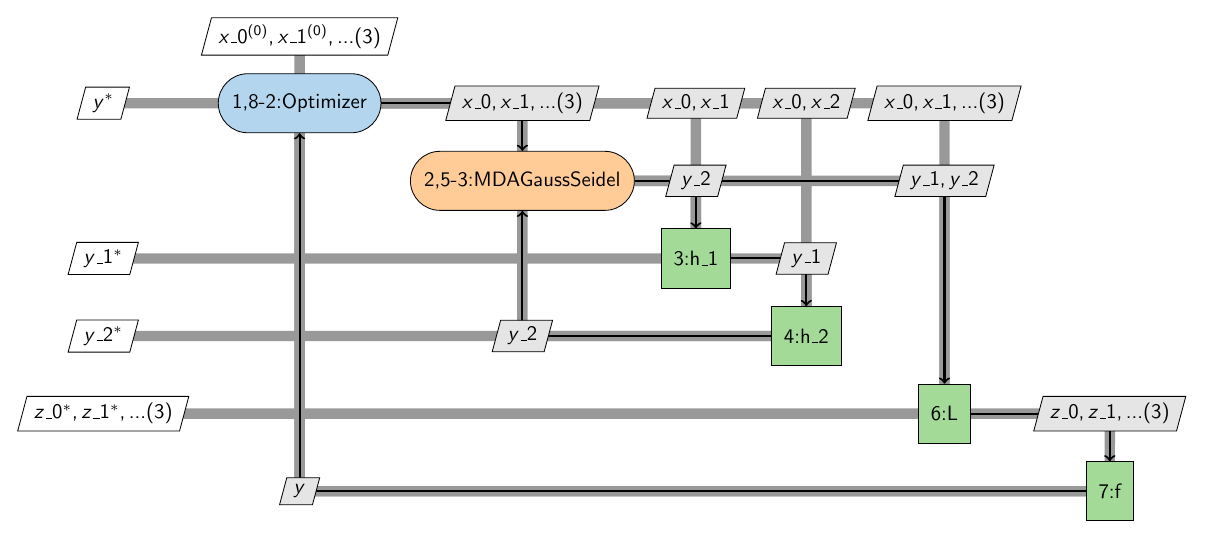}
        \caption{Gauss-Seidel algorithm.}
        \label{fig:xdsmgs}
    \end{subfigure}
     \begin{subfigure}[b]{0.75\textwidth}
        \includegraphics[width=\linewidth]{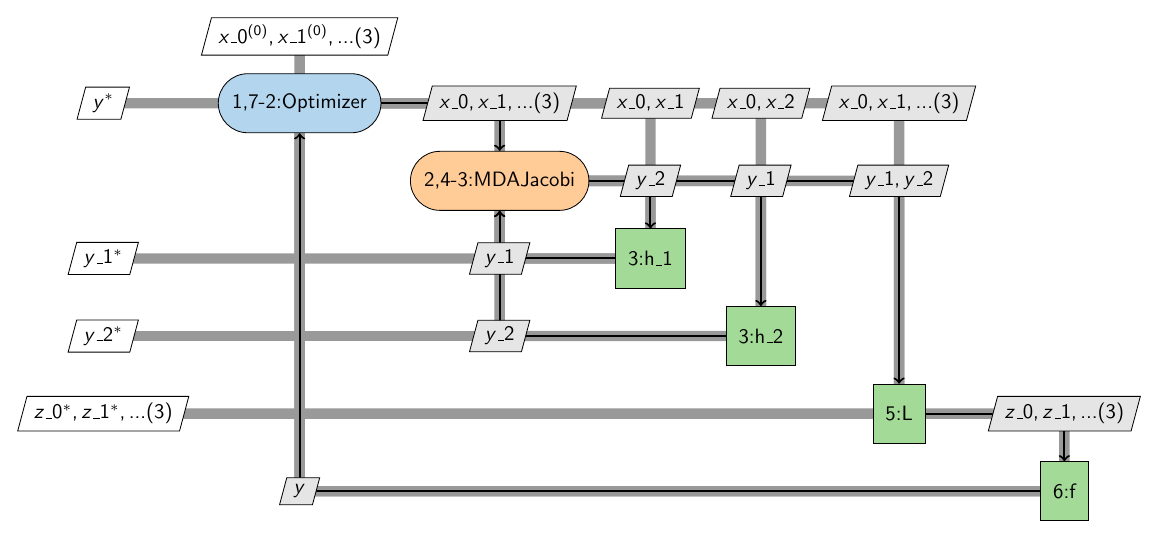}
        \caption{Jacobi algorithm.}
        \label{fig:xdsmj}
    \end{subfigure}
    \caption{Comparison of the MDF formulation using the Gauss-Seidel algorithm and the XDSM of the MDF formulation using the Jacobi one in terms of XDSM.}    
    \label{fig:xdsm}
\end{figure*}

First, we can see that both XDSMs have an orange block representing the algorithm managing the strong coupling between $h_1$ and $h_2$.
From the view of the process represented by the black line,
the optimizer provides a design vector $\bx^{(k)}$ to the coupling algorithm at the $k$-th iteration,
then the coupling algorithm is executed and provides a design vector $\by^{(k)}$ to the link discipline $L$,
which in turn provides its output value to the Rosenbrock function $f$.
The difference is in the coupling stage
where the disciplines $h_1$ and $h_2$ are executed sequentially in the case of the Gauss-Seidel algorithm
and in parallel in the case of the Jacobi algorithm.

\subsection{Benchmark problem 1 - Two disciplines with scalar design variables}\label{subsec:problem_1}

Problem \eqref{eq:rosenbrockmdo} is the first benchmark problem
and Table \ref{tab:results_pb1_1} and \ref{tab:results_pb1_2} summarize the results.
Table \ref{tab:results_pb1_1} displays the performance metrics of the Jacobi- and Gauss-Seidel-based optimization processes
in terms of mean and standard deviation estimated from 100 different values of the starting point $\bx^{(0)}$ 
constituting a design of experiments generated by LHS
and optimized by simulated annealing.

\begin{table}[!h]
    \centering
    \resizebox{\linewidth}{!}{
    \begin{tabular}{c|c|c|c|c|c|c|c|c|c|c}
          & $\Delta_{\bx}$ & $\Delta_f$ & $n_f$ & $n_f'$ & $n_L$ & $n_L'$ & $n_{h_1}$ & $n_{h_1}'$ & $n_{h_2}$ & $n_{h_2}'$ \\ 
         \hline 
         \multirow{2}{*}{J} & 0.0 & 0.0 & 48.7 & 48.0 & 48.7 & 48.0 & 490.8 & 48.0 & 490.8 & 48.0 \\ 
         & (0.0) & (0.0) & (12.4) & (12.5) & (12.4) & (12.5) & (129.8) & (12.5) & (129.8) & (12.5) \\ 
         \hline 
         \multirow{2}{*}{GS} & 0.0 & 0.0 & 48.6 & 47.8 & 48.6 & 47.8 & 347.8 & 47.8 & 300.0 & 47.8 \\ 
         & (0.0) & (0.0) & (12.3) & (12.4) & (12.3) & (12.4) & (93.9) & (12.4) & (81.9) & (12.4) \\   
    \end{tabular}
    }
    \caption{Performance of the Jacobi and Gauss-Seidel algorithms in the case of Problem 1.}
    \label{tab:results_pb1_1}
\end{table}

Firstly,
both optimization processes have converged to the theoretical solution $\bx^*$
as the estimation errors $\Delta_{\bx}$ and $\Delta_f$ are zero, whatever the initial value of $\bx$.

Secondly,
the number of evaluations of the functions $f$ and $L$ is essentially identical, 
both in terms of mean and standard deviation.
This is because the coupling algorithms have converged sufficiently finely, 
with a normalized residual norm\footnote{
Jacobi and Gauss-Seidel are iterative algorithms updating the value of $\by$ sequentially.
The first iteration updates the initial value $\by^{(0)}$ to $\by^{(1)}$
and the $i$-th iteration updates $\by^{(i-1)}$ to $\by^{(i)}$.
The normalized residual norm at this iteration is $\frac{\|\by^{(i)}-\by^{(0)}\|_2}{\|\by^{(0)}\|_2}$.
} 
lower than $10^{-6}$
for any point $\bx$ given by the optimization algorithm.
In this case,
$h(\bx,\by)\approx \by$ is verified,
$L(\bx,\by)\approx \bx$ is verified in turn
and therefore,
solving Problem \eqref{eq:rosenbrock} and Problem \eqref{eq:rosenbrockmdo} is equivalent
as stated in Proposition \ref{prop:main_result}.
A difference could arise if the Jacobi and Gauss-Seidel algorithms did not converge at certain iterations,
which would break this equivalence and open the door to different optimization paths.

Thirdly,
the number of evaluations of the derivatives of $f$, $L$, $h_1$ and $h_2$ is essentially identical for the same reason
and the fact that Jacobi and Gauss-Seidel algorithms do not use the derivatives of $h_1$ and $h_2$.

Lastly,
the focus is on the number of evaluations of $h_1$ and $h_2$
to find the most efficient coupling algorithm for this particular MDO problem.
Table \ref{tab:results_pb1_1} shows that 
this number is up to 64\% higher when using the Jacobi algorithm,
which means that the Gauss-Seidel algorithm is more suitable in terms of computational resources while ensuring high accuracy.
However,
in the event that a distributed computing infrastructure is available, 
the Jacobi algorithm would benefit from its parallelized structure
and the conclusions would be all the more reversed
that the number of processes would be important.
More precisely,
the number of evaluations of $h_1$, $h_2$ and their Jacobian functions $\nabla_{\bx} h_1$ and $\nabla_{\bx} h_2$
would be divided by the number of processes in the case of the Jacobi algorithm.

Based on these results,
Gauss-Seidel is better than Jacobi for this type of problem if calculations cannot be parallelized, and vice versa.
However, 
the number of evaluations of $h_1$ and $h_2$ appears to be significant on both sides
and it might be appropriate to reduce it using an acceleration method.
For example,
Table \ref{tab:results_pb1_2} compares these coupling algorithms using a polynomial extrapolation method \cite{acceleration}.
The statistics are very similar,
except for the number of evaluations $n_{h_1}$ and $n_{h_2}$ which have been reduced by 63\% and 64\% respectively in the case of Jacobi
and by 40\% and 37\% respectively in the case of Gauss-Seidel.
As for the number of evaluations of $h_1$ (resp. $h_2$),
it is 16\% (resp. 10\%) higher in the case of Gauss-Seidel than in the case of Jacobi,
which leads to the conclusion that Gauss-Seidel is better for this type of problem, with or without parallelization of calculations.

Finally,
this decrease in the number of evaluations leads us to adopt the minimum polynomial acceleration method for the following benchmark problems.

\begin{table}[!h]
    \centering
    \resizebox{\linewidth}{!}{
    \begin{tabular}{c|c|c|c|c|c|c|c|c|c|c}
          & $\Delta_{\bx}$ & $\Delta_f$ & $n_f$ & $n_f'$ & $n_L$ & $n_L'$ & $n_{h_1}$ & $n_{h_1}'$ & $n_{h_2}$ & $n_{h_2}'$ \\ 
         \hline 
         \multirow{2}{*}{J} & 0.0 & 0.0 & 48.2 & 47.5 & 48.2 & 47.5 & 180.9 & 47.5 & 175.7 & 47.5 \\ 
         & (0.0) & (0.0) & (11.9) & (12.0) & (11.9) & (12.0) & (42.1) & (12.0) & (41.7) & (12.0) \\ 
         \hline 
         \multirow{2}{*}{GS} & 0.0 & 0.0 & 48.4 & 47.6 & 48.4 & 47.6 & 209.5 & 47.6 & 193.5 & 47.6 \\ 
         & (0.0) & (0.0) & (12.1) & (12.1) & (12.1) & (12.1) & (51.1) & (12.1) & (48.2) & (12.1) \\         
    \end{tabular}
    }
    \caption{Performance of the Jacobi and Gauss-Seidel algorithms in the case of Problem 1, when using an acceleration method.}
    \label{tab:results_pb1_2}
\end{table}

\subsection{Benchmark problem 2 - Six disciplines with scalar design variables}\label{subsec:problem_2}

This second benchmark problem aims to show that the methodology is not limited to two strongly coupled disciplines.
For this purpose,
we will compare the Jacobi and Gauss-Seidel algorithms in dimension seven, 
using Problem \eqref{eq:rosenbrock} with $N=7$.
The resulting number of strongly coupled disciplines is six.
It is important to remind that in this case, 
the function has a local minimum at point $\tilde{\bx}=(-1,1,1,1,1,1,1)$
in addition to its global minimum at point $\bx^*=(1,1,1,1,1,1,1)$ \cite{Shang2006}.
Their values are $f(\tilde{\bx})=4$ and $f(\bx^*)=0$.
To avoid converging towards the local optimum,
the local optimization algorithm SLSQP is applied from 10 starting points $\bx^{(0)}$ constituting a design of experiments generated by LHS and the best solution is selected.
The execution of this multi-start SLSQP algorithm is repeated 100 times
using different seeds for the pseudo-random number generator of the LHS technique.

Table \ref{tab:results_pb2} shows the results.
First, we can see that the number of iterations of the optimization algorithm is the same for both formulations, namely 397.
Secondly,
the metric $\Delta_{\bx}$ and $\Delta_f$ shows a convergences towards the global minimum for both Jacobi and Gauss-Seidel.
Thirdly, 
the average number of evaluations of the strongly coupled disciplines, defined $\frac{1}{N}\sum_{i=1}^Nn_{h_i}$, in the case of Jacobi is higher than in the case of Gauss-Seidel by 56\%.
In the absence of parallelization, 
Gauss-Seidel is therefore cheaper than Jacobi.
Conversely, parallelization makes Jacobi more appropriate.
Finally, 
it should be noted that 
the number of evaluations is much higher than in the previous benchmark problem
due to the multi-start algorithm;
less greedy global algorithms might be more suitable.

\begin{table}[!h]
    \centering
    \begin{tabular}{c|c|c|c|c|c|c|c|c}
          & $\Delta_{\bx}$ & $\Delta_f$ & $n_f$ & $n_f'$ & $n_L$ & $n_L'$ & $\bar{n}_h$ & $\bar{n}'_h$ \\ 
         \hline 
         \multirow{2}{*}{J} & 0.0 & 0.0 & 397.2 & 352.5 & 397.2 & 352.5 & 3060.1 & 352.5 \\ 
         & (0.0) & (0.0) & (27.3) & (24.6) & (27.3) & (24.6) & (208.9) & (24.6) \\ 
         \hline 
         \multirow{2}{*}{GS} & 0.0 & 0.0 & 397.7 & 356.7 & 397.7 & 356.7 & 1956.8 & 356.7 \\ 
         & (0.0) & (0.0) & (27.2) & (25.6) & (27.2) & (25.6) & (134.6) & (25.6)\\ 
    \end{tabular}
    \caption{Performance of the Jacobi and Gauss-Seidel algorithms in the case of Problem 2.}
    \label{tab:results_pb2}
\end{table}

\subsection{Benchmark problem 3 - Two disciplines with vectorial design variables}\label{subsec:problem_3}

This third benchmark problem shows that the methodology applies to vectorial design variables,
using a multi-start SLSQP algorithm to deal with the local minimum.
For this purpose, 
we consider Problem \eqref{eq:rosenbrock} with $N=7$
where the design vector $\bx$ is decomposed as $\bx=(\bx_0,\bx_1,\bx_2)$ with $\bx_0\in[-2,2]^2$, $\bx_1\in[-2,2]^3$ and $\bx_2\in[-2,2]^2$.
Thus,
the corresponding MDO problem includes two strongly coupled disciplines $h_1$ and $h_2$.
Table \ref{tab:results_pb3} shows that the formulations are slightly equivalent in terms of number of iterations
and Gauss-Seidel requires 10\% more evaluations per iteration.

\begin{table}[!h]
    \centering
    \begin{tabular}{c|c|c|c|c|c|c|c|c|c|c}
          & $\Delta_{\bx}$ & $\Delta_f$ & $n_f$ & $n_f'$ & $n_L$ & $n_L'$ & $\bar{n}_h$ & $\bar{n}'_h$ \\ 
         \hline 
         \multirow{2}{*}{J} & 0.0 & 0.0 & 396.2 & 363.1 & 396.2 & 363.1 & 1754.0 & 363.1 \\ 
         & (0.0) & (0.0) & (27.1) & (26.0) & (27.1) & (26.0) & (122.1) & (26.0) \\ 
         \hline
         \multirow{2}{*}{GS} & 0.0 & 0.0 & 383.7 & 349.4 & 383.7 & 349.4 & 1871.5 & 349.4 \\ 
         & (0.0) & (0.0) & (25.1) & (24.1) & (25.1) & (24.1) & (124.4) & (24.1)\\ 
    \end{tabular}
    \caption{Performance of the Jacobi (J) and Gauss-Seidel (GS) algorithms in the case of Benchmark problem 3.}
    \label{tab:results_pb3}
\end{table}

\subsection{Benchmark problem 4 - Two disciplines with non-linear couplings}\label{subsec:problem_4}

This last benchmark problem shows that the methodology supports non-linear couplings.
For that purpose, it uses the coupling function defined in Equations \eqref{eq:nl1} and \eqref{eq:nl2} with the logistic functions $\sigma_{i,j}(x)=1/(1+e^{-0.3x})$.
Table \ref{tab:results_pb4} shows that all the optimizations have converged to the optimum $\bx^*$, as $\Delta_{\bx}$ and $\Delta f$ have a standard deviation and a mean of zero. 
The convergence rate is slower than in the linear case presented in Table \ref{tab:results_pb1_2}, with an average number of iterations of 70 instead of 48.
This means that the non-linear couplings made the task complex for MDO algorithms.

\begin{table}[!h]
    \centering
    \resizebox{\linewidth}{!}{
    \begin{tabular}{c|c|c|c|c|c|c|c|c|c|c}
          & $\Delta_{\bx}$ & $\Delta_f$ & $n_f$ & $n_f'$ & $n_L$ & $n_L'$ & $n_{h_1}$ & $n_{h_1}'$ & $n_{h_2}$ & $n_{h_2}'$ \\ 
         \hline 
         \multirow{2}{*}{J} & 0.0 & 0.0 & 70.9 & 69.5 & 70.9 & 69.5 & 362.0 & 139.1 & 362.0 & 139.1 \\ 
         & (0.0) & (0.0) & (17.6) & (17.7) & (17.6) & (17.7) & (89.2) & (35.4) & (89.2) & (35.4) \\ 
         \hline 
         \multirow{2}{*}{GS} & 0.0 & 0.0 & 70.0 & 69.2 & 70.0 & 69.2 & 428.8 & 138.4 & 428.8 & 138.4 \\ 
         & (0.0) & (0.0) & (18.2) & (18.2) & (18.2) & (18.2) & (112.8) & (36.5) & (112.8) & (36.5) \\  
    \end{tabular}
    }
    \caption{Performance of the Jacobi and Gauss-Seidel algorithms in the case of Problem 4.}
    \label{tab:results_pb4}
\end{table}

\section{Conclusion}\label{sec:conclusion}

Benchmark problems with theoretically known solutions are a cornerstone of algorithm assessment in optimization, yet they remain scarce in the multidisciplinary design optimization setting. To address this gap, we proposed a systematic methodology to transform any mono-disciplinary optimization problem into a family of parametric MDO problems sharing the same solution.

The construction relies on two ingredients provided by the user: a set of coupling equations that introduce interdisciplinary dependencies, and a link function that eliminates the coupling variables and recovers the original mono-disciplinary problem. Theoretical conditions ensuring the equivalence between the two problems were established: equivalence is guaranteed for any continuous coupling functions, possibly non-linear, when the space of coupling variables is a convex compact set.
As a consequence, any mono-disciplinary problem with a known solution gives rise to an MDO problem with the same solution.

Because the methodology is agnostic to the number of disciplines and variable dimensions, it is naturally suited for scalability studies. A specialized version for linear coupling equations with constant coefficients was also developed, for which the solution is explicit and the coefficients can be sampled independently at random. This randomization enables statistical evaluation of MDO algorithm performance through indicators such as mean and standard deviation over multiple problem instances.

We illustrated the versatility of the approach by constructing scalable MDO Rosenbrock problems in various configurations — scalar and vectorial design variables, varying numbers of disciplines — and used them to compare the Jacobi and Gauss-Seidel coupling algorithms. Beyond this illustration, the proposed framework opens a systematic route to generating MDO benchmarks of arbitrary scale and complexity from the extensive catalog of existing mono-disciplinary test problems, paving the way for more rigorous and reproducible comparisons of MDO algorithms.


\section*{Acknowledgements}
The authors would like to thank François Gallard for his comments, which helped improve the scientific quality of the manuscript through his expertise in MDO.

\bibliographystyle{plain}
\bibliography{references}

\end{document}